\documentclass[11pt]{amsart}

\usepackage[utf8]{inputenc}
\usepackage{fullpage, mathrsfs, mathtools, amssymb, tikz,bm, amscd, scrextend, tikz-cd, verbatim, enumerate, tabu}
\usepackage[inline]{enumitem}
\usepackage{microtype}
\usepackage{verbatim}
\usepackage{stmaryrd}
\usepackage[colorlinks=true, linkcolor=red!80!black, urlcolor=purple, citecolor=blue!70!black]{hyperref}

\usetikzlibrary{matrix, arrows}
\input xy
\xyoption{all} 

	\newtheorem{theorem}{Theorem}[section]
	
	\newtheorem{prop}[theorem]{Proposition}
	\newtheorem{cor}[theorem]{Corollary}

	\theoremstyle{definition}
	\newtheorem{definition}[theorem]{Definition}

\newtheorem{remark}[theorem]{Remark}
\newtheorem{notation}[theorem]{Notation}

\theoremstyle{remark}

\newcounter{step}

\newcommand{\Q}{{\mathbb Q}}

\newcommand{\bP}{{\mathbb P}}

\newcommand{\bQ}{{\mathbb Q}}
\newcommand{\bF}{{\mathbb F}}

\newcommand{\calO}{\mathcal O}
\newcommand{\calM}{\mathcal M}

\newcommand{\calE}{\mathcal E}

\newcommand{\calC}{\mathcal C}

\newcommand{\calR}{\mathcal R}
\newcommand{\calQ}{\mathcal Q}

\newcommand{\calD}{\mathcal D}

\newcommand{\mb}[1]{\mathbb{#1}}

\title{Moduli of double covers and degree one del Pezzo surfaces} 
\author{Kenneth Ascher \& Dori Bejleri}

\begin{document}

\maketitle

\begin{abstract}
Given a degree one del Pezzo surface with canonical singularities, the linear series generated by twice the anti-canonical divisor exhibits the surface as the double cover of the quadric cone branched along a sextic curve. It is natural to ask if this description extends to the boundary of a compactification of the moduli space of degree one del Pezzo surfaces. The goal of this paper is to show that this is indeed the case.  In particular, we give an explicit classification of the boundary of the moduli space of anti-canonically polarized broken del Pezzo surfaces of degree one as double covers of degenerations of the quadric cone. 
\end{abstract}

\section{Introduction}
The anti-canonical linear series of del Pezzo surfaces have a rich geometric structure. For degree one del Pezzo surfaces, the linear series $|-K_X|$ is a pencil of elliptic curves with a unique base point, and the blowup of this base point gives a rational elliptic surface with a section. On the other hand, the linear system $|-2K_X|$ exhibits the surface $X$ as a double cover of $\bP(1,1,2)$, i.e. the quadric cone, branched along a sextic curve. At the same time, understanding modular compactifications of the space of degree one del Pezzo surfaces is a problem with a long history (see e.g. \cite{ishii, sekiguchi, miranda, hl, oss, ade}). In \cite{rational}, we construct and describe an explicit modular compactification $\calR$ of this space using the theory of stable pairs. The boundary of this space parametrizes anti-canonically polarized broken del Pezzo surfaces of degree one -- slc surfaces $X$ such that $K_X$ is anti-ample and $K_X^2 = 1$ (see \cite[Theorem 1.1]{rational}). 

In light of this, it is natural to ask whether the description of a degree one del Pezzo surface as a double cover of $\bP(1,1,2)$ extends along the boundary of this moduli space. That is, are anti-canonically polarized broken del Pezzo surface of degree one also double covers of some degenerations of $\bP(1,1,2)$? The main goal of this paper is to show that this is indeed the case.

Let $\calD^{1,s} \subset \calR$ denote the moduli stack of degree one del Pezzo surfaces with canonical singularities. If $\calQ_3$ denotes the stack of pairs $(Q,C)$, where $Q$ is a quadric cone in $\bP^3$ and $C \subset Q$ is a complete intersection of $Q$ with a cubic, then there exists a smooth and separated substack $\calQ^s_3 \subset \calQ_3$ and a map $\calD^{1,s} \to \calQ^s_3$ which is an isomorphism up to taking the relative coarse moduli space (see Corollary \ref{cor:substack}). In Section \ref{sec:stablepairs} we will define a suitable compactification $\calQ^s_3 \subset \overline{\calQ}$ parametrizing pairs on a singular quadric cone in $\bP^3$ (see Definition \ref{def:compactification}).

\begin{theorem}[see Theorem \ref{thm:maintheorem}]\label{thm:main}
There is a separated morphism $\calR \to \overline{\calQ}$ which extends the morphism $\calD^{1,s} \to \calQ^s_3$. This morphism is induced by a double cover map from the universal family of broken del Pezzo surfaces of degree one to a family of quadric cones branced along a family of complete intersections with a cubic. Moreover, the induced morphism on the relative coarse moduli space $\calR^c \to \overline{\calQ}$ is a monomorphism. 
\end{theorem}

In \cite[Section 7.1]{rational}, we gave an explicit classification of the broken del Pezzo surfaces on the boundary of $\calR$. In this paper we explicitly show that the linear series $|-2K_X|$ on a broken degree one del Pezzo induces a double cover map $X \to Q$ of a quadric cone $Q \subset \mathbb{P}^3$, where now $Q = \mathbb{P}^2 \cup_l \mathbb{P}^2$ is the cone over two lines. We note that the proof of Theorem \ref{thm:doublecover} is constructive, in that we show explicitly how each surface on the boundary of $\calR$ can be realized as a double cover of $Q$ and that the double cover is induced by $|-2K_X|$ (see Remark \ref{rmk:table} and Table \ref{table:fibers}).

We note that Deopurkar-Han \cite{DH} recently studied the moduli space of $(3,3)$ curves on $\bP^1 \times \bP^1$ and their $\bQ$-Gorenstein smoothings. Their space contains a boundary divisor corresponding to pairs of a quadric cone with a sextic curve which is birational to $\calR$ by Theorem \ref{thm:main}.\\

%There are some questions that naturally arise from this work, and we consider these problems for future investigation.  For instance one can study how the constructions given here compare to the moduli space of $(3,3)$ curves on $\bP^1 \times \bP^1$ and their $\bQ$-Gorenstein smoothings, which was recently studied by Deopurkar-Han \cite{DH}. This space contains a boundary divisor corresponding to pairs of a quadric cone with a sextic curve which is birational to $\calR$ by Theorem \ref{thm:main}. One can also consider the viewpoint of K-stability (using the framework of \cite{ADL}; see also \cite{oss}).  

We work over an algebraically field of characteristic zero.

%By Theorem \ref{thm:main}, one can understand moduli of degree one del Pezzo surfaces as moduli of sextics in $\bP(1,1,2)$. Following the framework established in \cite{ADL} to study K-moduli spaces of log Fano pairs, it is natural to study K-moduli spaces of log Fano pairs $(\bP(1,1,2), \alpha C_6)$, where $C_6$ is a sextic curve in $\bP(1,1,2)$ and $\alpha \in \mathbb{Q}_{>0}$. It would be interesting to understand how varying the parameter $\alpha$ could be used to interpolate between various compactifications of the moduli space of degree one del Pezzo surfaces. In this setting, it is also natural to ask how our work compares to the constructions of Odaka-Spotti-Sun \cite{oss}. 

%As $\bP(1,1,2)$ admits a $\bQ$-Gorenstein smoothing to $\bP^1 \times \bP^1$, one can ask how our constructions compare to the moduli spaces of $(3,3)$ curves on $\bP^1 \times \bP^1$ and their $\bQ$-Gorenstein deformations, which was recently studied by Deopurkar-Han \cite{DH}.   

\subsection*{Acknowledgements} We thank J\'anos Koll\'ar for asking us this question and for helpful conversations. This work was partially completed while the authors were in residence at MSRI in Spring 2019 (NSF No. DMS-1440140). Both authors supported in part by NSF Postdoctoral Fellowships.

\section{Moduli of degree one del Pezzo surfaces}\label{sec:background}

We recall some background on del Pezzo surfaces of degree one with canonical singularities. 

\begin{definition} A degree $n$ del Pezzo surface is a surface $X$ with canonical singularities such that $-K_X$ is ample and $K_X^2 = n$. 
\end{definition} 

For $n \geq 2$, the divisor $-K_X$ is very ample and the linear series $|-K_X|$ embeds $X$ into $\mathbb{P}^n$. In the degree $n = 1$ case, which is the focus of this paper, the linear series $|-K_X|$ is a pencil of elliptic curves with a unique base point. On the other hand, the following is well known: 

\begin{prop}\label{prop:-2K} For degree $n =1$ the morphism induced by the linear system $|-2K_X|$,
$$
\varphi_{|-2K_X|} : X \to \mb{P}^3
$$
is basepoint free and exhibits $X$ as a double cover of a quadric cone branched along the complete intersection with a cubic surface. 
\end{prop} 

If we write $H^0(-K_X) = \langle x, y \rangle$, then one can readily compute using Kodaira vanishing and Riemann-Roch that $h^0(-2K_X) = 4$ so that
$
H^0(-2K_X) = \langle x^2, xy, y^2,z\rangle
$
for some section $z$. This induces a map to $\mathbb{P}^3$ with image given by $x_0x_2 - x_1^2 = 0$.

\begin{remark}We note that interpretations of $|-K_X|$ and $|-2K_X|$ as an elliptic pencil and double cover respectively are also known to be true in the context of del Pezzo surfaces of degree one with \emph{canonical singularities}. For the proofs of these statements we refer the reader to the work of Kosta \cite[Section 2.1]{kosta}.\end{remark}

We now show that the construction given in Proposition \ref{prop:-2K} can also be done in families. Before doing so, we set up some notation. If $\pi: X \to T$ desnotes a flat family of anti-canonically polarized degree one del Pezzo surfaces, we let $\mathbb{P} = \mathbb{P}(\calE)$ denote the $\mathbb{P}^3$ bundle given by $\calE := \pi_*\calO_X(-2K_{X/T})$. 

\begin{prop}\label{prop:cover} Let $\pi : X \to T$ be a flat family of anti-canonically polarized degree one del Pezzo surfaces. Then the line bundle $\calO_X(-2K_{X/T})$ is globally generated over $T$, and the induced morphism $\varphi : X \to \mathbb{P}$ is a $2$-to-$1$ map onto a family $Q \to T$  of singular quadrics embedded in  $\mathbb{P}$ over $T$, branched over a family of complete intersections with a cubic. 
\end{prop}

\begin{proof} By Kodaira vanishing, $H^1(X_t, \calO_{X_t}(-2K_{X_t})) = 0$ for each $t \in T$, and so by cohomology and base change, the vector bundle $\calE = \pi_*\calO_X(-2K_{X/T})$ is locally free of rank four and compatible with base change. Consider the map 
$
\pi^*\calE \to \calO_X(-2K_{X/T}).
$
Since $-2K_{X_t}$ is globally generated on the fiber $X_t$, this map is fiberwise surjective which implies global generation over $T$. This induces a $T$-morphism $X \to \mathbb{P} = \mathbb{P}(\calE)$ which fiberwise is a double cover of a quadric cone as above. 
\end{proof} 

As a corollary, we can identify the moduli space of degree one del Pezzo surfaces with an open substack of the moduli space of pairs $(Q, C)$ where $Q$ is a quadric cone in $\mathbb{P}^3$ and $C \subset Q$ is a complete intersection of $Q$ with a cubic. 

\begin{notation}\label{not:moduli}Let  $\calD^1$ denote the moduli space of anti-canonically polarized del Pezzo surfaces of degree one with ADE singularities, and let $\calQ_3$ denote the moduli space of pairs $(Q,C)$ of a (not necessarily integral) quadric cone $Q$, i.e. $Q$ can be the cone over two lines or a double line, and $C \subset Q$ a complete intersection curve with a cubic. Finally, we let $\calQ_3^{ADE}$ denote the open substack of pairs where $Q$ is normal, $C$ avoids the cone point, and $C$ has at worst ADE singularities (see Table \ref{table:ADE}). 
 \end{notation}

Thus by Proposition \ref{prop:cover}, there is a natural map $\calD^1 \to \calQ_3$. 

\begin{cor} Let $\calD^{1,c} \to \calQ_3$ be the relative coarse moduli space of natural map $\calD^1 \to \calQ_3$ above. There is an isomorphism of $\calD^{1,c}$ with  $\calQ_3^{ADE}$. \end{cor}

\begin{proof} By Proposition \ref{prop:cover}, the set theoretic map is actually a morphism on the level of moduli spaces. For any surface $X$, the double cover structure is given by the elliptic involution on the anti-canonical curves of $X$. In particular, the branch locus of $X \to Q$ is given by the torsion points of the anti-canonical curves. There is a fixed torsion point for each anti-canonical curve at the basepoint of $|-K_X|$, the identity of the group structure, and the other torsion points are disjoint. Thus $C$ is disjoint from the cone point lying under the basepoint of $|-K_X|$. Moreover, $X$ has ADE singularities if and only if the branch curve $C$ does. Then to conclude, one just observes that given a family $(\calQ, \calC) \to T$, the double cover of $\calQ$ branched along $\calC$ is a family of degree one del Pezzo surfaces so that the map above has an inverse. Thus the map $\calD^{1,c} \to \calQ_3^{ADE}$ is a representable bijection and so it is an isomorphism by Zariski's main theorem for stacks, see e.g. \cite[Theorem A.5]{ai}. \end{proof}

\begin{remark} In the corollary above, one wants to say that $\calD^1$ is identified with $Q_3^{ADE}$. The issue is that the del Pezzo surface has an extra involution given by swapping the fibers of the double cover map and so $\calD^{1,c} \to \calQ_3$ can be thought of as a $\mu_2$-gerbe over its image. 
\end{remark}

\begin{table}[hpt!]
\centering
\caption{ADE curve singularities}\label{table:ADE}
    $\begin{tabu}{|l|c|}
    \hline
    \text{Singularity type} & \text{Local equation}\\
    \hline
    \mathrm{A}_n (n \geq 1) & x^2 + y^{n+1} = 0   \\
    \mathrm{D}_n (n \geq 4) & y(x^2 + y^{n -2}) = 0   \\ 
    \mathrm{E}_6 & x^3 + y^4 = 0  \\
    \mathrm{E}_7 & x(x^2 + y^3) = 0 \\
    \mathrm{E}_8 & x^3 + y^5 = 0  \\
    \hline
    \end{tabu}$ 
\end{table}

Let us explicate the correspondence between the geometry of $X$ and the geometry of the pair $(Q,C)$. The anti-canonical curves of $X$ correspond to the lines on $Q$. The singular anti-canonical curves, or equivalently the singular fibers of the corresponding rational elliptic surface, correspond to lines $l$ that meet $C$ in multiplicity at least two, i.e. the tangent lines to $C$. One can check that there are $12$ such lines, counted with multiplicity. The curve $C$ is a tri-section of the corresponding rational elliptic surface and in particular it is a trigonal  curve of genus four. The Kodaira fiber type of an anti-canonical curve in the corresponding rational elliptic surface can be read off from the the line $l$ and the singularities of $C$ according to Table \ref{table:fibers}. 

\begin{table}[htp!]
    \centering
    \caption{Kodaira fiber types from the branch locus}\label{table:fibers}
    \begin{tabu}{|c|c|c|}
    \hline
    \text{Singular fiber} &\text{multiplicity of} $l \cap C$ & \text{Singularity of} $C$\\
    \hline
       $\mathrm{I}_0$  & transverse &  none (smooth)  \\
       $\mathrm{I}_1$  &   $2$ & none (smooth) \\
       $\mathrm{I}_n$, $n \geq 2$  &  $2$ & $\mathrm{A}_{n-1}$  \\
       $\mathrm{II}$ &  $3$ & none (smooth) \\
       $\mathrm{III}$ & $3$ & $\mathrm{A}_1$  \\
       $\mathrm{IV}$ &   $3$ &$\mathrm{A}_2$  \\
       $\mathrm{I}_n^*$ &  $2$ & $\mathrm{D}_{n + 4}$  \\
       $\mathrm{IV}^*$ & $2$ &$\mathrm{E}_6$  \\
       $\mathrm{III}^*$ &  $2$ &$\mathrm{E}_7$  \\
       $\mathrm{II}^*$ &  $2$ & $\mathrm{E}_8$  \\
       \hline

    \hline
    \end{tabu}
\end{table}

The stack $\calD^1 \cong \calQ^{ADE}_3$ is \emph{not} a separated Deligne-Mumford stack. A separated open Deligne-Mumford substack $\calD^{1,s} \subset \calD^1$ is cut out by GIT stability computed by Miranda \cite{miranda}: $X$ is GIT stable if and only if it has at worst $\mathrm{A}_n$ singularities. As a result we obtain the following. 

\begin{cor}\label{cor:substack} The open substack $\calQ^s_3 \subset \calQ^{ADE}_3$ of pairs $(Q,C)$ such that $C$ avoids the cone point and has at worst $\mathrm{A}_n$ singularities is a smooth and separated Deligne-Mumford stack. 
\end{cor}

\section{Stable pair compactifications}\label{sec:stablepairs}

In this section we set up the compactifications of the moduli spaces of interest using the theory of stable pairs (see e.g. \cite{kollarbook}).  A natural choice of divisor on a degree one del Pezzo surface is the sum of the rational curves in the anti-canonical pencil, which as we saw above are exactly the nodal and cuspidal genus one curves lying over lines of $Q$. We can take these curves $F_i$ counted with multiplicity $m_i$ to endow $X$ with a boundary divisor 
$$
F = a\sum m_i F_i.
$$
Here we will take $a = \frac{1}{12} + \epsilon$ for $1 \gg \epsilon > 0$. This is the smallest $a$ for which the pair $(X,F)$ is stable as there are $12$ anti-canonical curves counted with multiplicity so that $F \sim_\Q -12aK_X$.

In \cite{rational}, we described the closure of the moduli space $\calD^{1,s}$ inside the moduli space of stable pairs using degenerations of rational elliptic fibrations; we denoted this space by $\calR := \calR(\frac{1}{12} + \epsilon)$. The choice of notation is suggestive -- $\calR$ was chosen since degree one del Pezzo surfaces are the blowdown of the section of a rational elliptic surface and this compactification is furnished by degenerating the elliptic fibration structure.

\begin{theorem}\cite[Theorem~1.1]{rational} There exists a proper Deligne-Mumford stack $\calR = \calR(\frac{1}{12} + \epsilon)$
parametrizing anti-canonically polarized broken del Pezzo surfaces of degree one with the following
properties:
\begin{itemize}
\item The interior $\calD^{1,s} \subset \calR$ parametrizes degree one del Pezzo surfaces with at worst rational double
point singularities.
\item The complement $\calR \setminus \calD^{1,s}$ consists of a unique boundary divisor $\partial \calR$ parametrizing 2-Gorenstein
semi-log canonical surfaces with ample anti-canonical divisor and exactly two irreducible
components.
\item The locus $\calR^{\circ} \subset \calR$ parametrizing surfaces such that every irreducible component is normal
is a smooth Deligne-Mumford stack.
\end{itemize}
\end{theorem}

On the other hand, we can consider marking the quadric cone $Q$ with lines corresponding to the singular members $F_i$. These are the lines that intersect $C$ with multiplicity $m \geq 2$. There are $12$ such lines counted with multiplcity $m - 1$ and so we consider
$$
L = a\sum (m_i - 1)l_i
$$
where again $a = \frac{1}{12}+ \epsilon$. Using a similar calculation, one can check that this is the smallest $a$ for which $(Q, L)$ is stable. Using this choice of boundary divisor, we can map $\calQ^s_3$ into a certain moduli space of stable pairs $\calM_v$, by taking $(Q,C)$ to the stable pair $(Q,L)$. Note here that the underlying surfaces are all isomorphic but the curve $C$ determines the boundary $L$. Let us denote this by $\varphi : \calQ^s_3 \to \calM_v$. 

\begin{definition}\label{def:compactification} Let $\overline{\calQ}$ be the closure of the graph of $\varphi: \calQ^s_3 \to \calM_v$ inside the product $\calQ_3 \times \calM_v$. 
\end{definition}

It is clear that $\calQ_3^s$ is a dense open substack of $\overline{\calQ}$. Moreover, $\overline{\calQ}$ parametrizes triples $(Q, C, L)$ where $Q$ is a singular quadric cone in $\mathbb{P}^3$, the curve $C$ is a complete intersection with a cubic, and $L = a\sum (m_i - 1)l_i$ where $\{l_i\}$ is a collection of lines on $Q$ meeting at a fixed point $p \in Q$ and intersecting $C$ at multiplicity $m_i$. Moreover, these triples are such that $(Q,C,L)$ is the central fiber $(Q_0, C_0, L_0)$ of a family $(Q_t, C_t, L_t)$ for $t \in T$ a smooth curve where $(Q_t, C_t)$ for $t \neq 0$ is a family of GIT stable pairs $\calQ^s_3$, and $(Q_t, L_t)$ is a family of stable pairs in $\calM_v$. Note that a priori, $\overline{\calQ}$ need not be proper nor Deligne-Mumford. 

\section{Broken degree one del Pezzo surfaces} 

In this section we will analyze the double cover structure on the broken degree one del Pezzo surfaces appearing on the boundary of $\calR$. Let us recall the description of these surfaces from \cite[Section 7.1]{rational}. For more details on some of the terminology (e.g. twisted fibers) we refer the readers to \cite{calculations, tsm, master}. 

\begin{theorem}\cite[Section 7]{rational}\label{thm:recallboundary} The boundary $\partial \calR = \calR \setminus \calD^{1,s}$ parametrizes surfaces of the following types:
\begin{enumerate}

\item[($B_{\mathrm{I}}$)] the slc union $X = Y_1 \cup_G Y_2$ where $Y_1$ and $Y_2$ are rational pseudoelliptic surfaces glued along a twisted $\mathrm{I}_0^*$ pseudofiber $G$ such that $-2K_X$ is Cartier and ample, and all pseudofibers of $X$ away from the double locus $G$ are of Kodaira type $\mathrm{I}_n, \mathrm{II}, \mathrm{III}$ or $\mathrm{IV}$;

\item[($B_{\mathrm{II}}$)] the slc union $X = Y_1 \cup_G Y_2$ where $Y_1$ is a rational pseudoelliptic surface with a twisted $\mathrm{I}_n^*$ pseudofiber at $G$, and all other fibers as above, $Y_2$ is an isotrivial $j$-invariant $\infty$ pseudoelliptic surface of type $2\mathrm{N}_1$ with $G$ a twisted $\mathrm{N}_1$ fiber, and $-2K_X$ Cartier and ample;

\item[($B_{\mathrm{III}}$)] the slc union $X = Y_1 \cup_G Y_2$ where $Y_1$ and $Y_2$ are both isotrivial $j$-invariant $\infty$ pseudoelliptic surfaces of type $2\mathrm{N}_1$ glued along twisted $\mathrm{N}_1$ fibers, and $-2K_X$ is Cartier and ample. 

\end{enumerate}

\end{theorem}

We will call these types of surfaces broken degree one del Pezzo surfaces of type $B_{\mathrm{I}}, B_{\mathrm{II}}$, and $B_{\mathrm{III}}$ respectively. We note that in \cite{rational}, we called the surfaces in $\calD^{1,s} \subset \calR$ surfaces of type A, and we referred to the type $B_{\mathrm{I}}$ surfaces as surfaces of type B. We did not give names to what appear here as type $B_{\mathrm{II}}$ or type $B_{\mathrm{III}}$.

The main theorem of this section is that, as in the case of degree one del Pezzo surfaces, the linear series $|-2K_X|$ on a broken degree one del Pezzo induces a double cover map $X \to Q$ of a quadric cone $Q \subset \mathbb{P}^3$, except now $Q = \mathbb{P}^2 \cup_l \mathbb{P}^2$ is the cone over two lines. 

\begin{theorem}\label{thm:doublecover} 
Let $Q = \mathbb{P}^2 \cup_l \mathbb{P}^2$ denote the quadric cone over two lines. If $X$ is a broken degree one del Pezzo surface, then there is a double cover $\varphi : X \to Q \subset \mathbb{P}^3$ branched along the double locus $l$ as well as a complete intersection curve $C \subset Q$ with a cubic. Moreover, if $H$ is a hyperplane section of $Q$, then $\varphi^*H \in |-2K_X|$ is a $2$-anti-canonical curve. 

\end{theorem} 

\begin{proof} We will construct the $Y_i$ appearing as components of type $B_{\mathrm{I}}$, $B_{\mathrm{II}}$ or $B_{\mathrm{III}}$ surfaces explicitly as double covers of $\mathbb{P}^2$ branched over a line and a cubic and then glue them together to obtain a map $Y_1 \cup_G Y_2 \to \mathbb{P}^2 \cup_l \mathbb{P}^2$. 

Consider the data of $(\bP^2, C, l, p)$ where $C \subset \bP^2$ is a cubic curve, $l \subset \bP^2$ is a line, and $p \in l$ is a point on the line avoiding $C$. Let $g : Y \to \bP^2$ be the double cover branched along $C \cup l$. We will see that each of the surfaces $Y_i$ as above are constructed in this way depending on how $C$ and $l$ meet. 

Indeed, let us first check that $Y$ is a rational pseudoelliptic surface. Since $Y$ is branched over a quartic curve in $\bP^2$, we can see that $-K_Y$ is ample by Riemann-Hurwitz. Therefore, $Y$ is a rational surface. Moreover, $Y$ has a pencil of elliptic curves through $q = g^{-1}(p)$ given by  pulling back the pencil of lines through $p$. If we blow up the ideal of $p$  as well as its inverse image,  we obtain a double cover $g' : Y' \to \bF_1$ where $Y'$ is elliptically fibered with section. Moreover, $g'$ is branched along the fiber $G'$ lying over the strict transform of $l$ in $\bF_1$. In particular, $G'$ supports a twisted fiber with non-reduced multiplicity two. If $S \subset Y'$ is the section, a local computation shows that in fact $Y'$ meets $G$ in an $\mathrm{A}_1$ singularity. Thus $Y$ is a rational pseudoelliptic surface with a pseudofiber $G$ lying over $l$ that supports a twisted pseudofiber with non-reduced multiplicity. 

The types of surfaces appearing in the above Theorem \ref{thm:recallboundary} depend on the singularities of $C$, and the intersection $C \cap l$. To obtain a pseudoelliptic component $Y$ as in surfaces of Type $B_{\mathrm{I}}$, we take $C$ to intersect $l$ transversely in three smooth points, with $C$ having at worst $\mathrm{A}_n$ singularities away from  $l \cap C$. In this case we obtain a surface with at worst $\mathrm{A}_n$ singularities away from the fiber $G$ which indicates that the pseudofibers are of type $\mathrm{I}_n$, $\mathrm{II}$, $\mathrm{III}$ or $\mathrm{IV}$ away from $G$. Moreover, the pesudofiber $G$ comes from the fiber $G'$ on $Y'$ supporting a multiplicity two pseudofiber along which $Y'$ has four $\mathrm{A}_1$ singularities, exactly an $\mathrm{I}_0^*$ fiber (see \cite[Lemma 4.2 (iii)]{calculations}). 

In surfaces of type $B_{\mathrm{II}}$, we have a normal component $Y$ similar to the one constructed above except the pseudofiber $G$ has Kodaira type $\mathrm{I}_n^*$. When $n = 1$, there is an $\mathrm{A}_3$ singularity on $Y$ along $G$ and this corresponds to $C$ being tangent along $l$ with multiplicity $2$ so that so that $C \cup l$ has an $\mathrm{A}_3$ singularity. If $n \geq 2$, then $Y$ has a $\mathrm{D}_{n + 2}$ singularity along $F$, and this corresponds to $C$ degenerating to an $\mathrm{A}_{n - 1}$ singularity where it meets $l$ so that $C \cup l$ has a $\mathrm{D}_{n + 2}$ singularity. In each of these cases, there is another $\mathrm{A}_1$ singularity on $Y$ along $G$ so that there must be another point at which $l$ and $C$ meet transversely. In particular, $l$ meets $C$ at multiplicity $2$ along the singular point since it meets $C$ multiplicity three in total. 

Finally, we have the non-normal isotrivial $j$-invariant $\infty$ components appearing in type $B_{\mathrm{II}}$ and $B_{\mathrm{III}}$ surfaces. To construct these surfaces, note that a double cover branched along a nonreduced curve of multiplicity two with smooth support is an slc surface with ramification divisor being the double locus. In this case, we take the branch curve $C$ degenerating to the union of a line $r$ and a double line $s$. Then on $Y$, the pencil of lines lifts to a pencil of nodal cubics with the node lying over $s$, and the other two torsion point lying over $r$. At the point where $s$ and $r$ intersect, the line in the pencil lifts to a Weierstrass $\mathrm{N}_1$ cusp, and over $l$, we obtain a multiplicity two twisted fiber with with an $\mathrm{A}_1$ singularity where $l$ meets $r$ and a singularity analytically isomorphic to a pinch point over the point where $s$ meets $l$. This is precisely the twisted $\mathrm{N}_1$ fiber. 

This shows that each component of a surface from Theorem \ref{thm:recallboundary} is obtained as the double cover of $\bP^2$ branched along the union $l \cup C$ of a line and a cubic curve $C$. Moreover, the elliptic pencil $Y$ is pulled back from the pencil of lines through a fixed point $p \in l$ avoiding $C$, and the pseudofiber $G$ along which the components are glued appearing as the ramification divisor lying over $l$.  Now we wish to obtain $Y_1 \cup_G Y_2$ as a double cover of $\bP^2 \cup_l \bP^2$ by gluing. The collection of data $(\bP^2, C, l, p)$ that determines $Y$ also determines four special points counted with multiplicity on the line $l$, namely the points $p$ and $l \cap C$. As a divisor on $l$, the intersection $l \cap C$ is precisely the different $\mathrm{Diff}_l(C)$. 

To glue the double covers given by two such collections of data $(\bP^2, C_1, l_1, p_1)$ and $(\bP^2, C_2, l_2, p_2)$, we must pick an isomorphism $\tau : l_1 \to l_2$ such that $\tau(p_1) = p_2$ (in order to identify the elliptic pencils on $Y_1$ and $Y_2$) and such that $\tau(\mathrm{Diff}_{l_1}(C_1)) = \mathrm{Diff}_{l_2}(C_2)$. Note in particular, such a $\tau$, if it exists, is unique and $\tau$ exists if and only if the $j$-invariant of the four special points on each of the lines agrees. Indeed this is the $j$-invariant of the pseudofiber $G$ on each of the components $Y_i$ and must agree if the $Y_i$ glue along $G$ to form a surface of type $B_{\mathrm{I}}$, $B_{\mathrm{II}}$, or $B_{\mathrm{III}}$. The necessity of the latter condition on the different comes from Koll\'ar's gluing theory (see \cite[Chapter 5]{singmmp}): this is precisely the condition so that the branch locus descends to a $\bQ$-Cartier divisor on $\bP^2 \cup_l \bP^2$ glued by this $\tau$. 

Now we have a map $Y_1 \sqcup Y_2 \to \bP^2 \cup_l \bP^2$ by composing the double cover map on each component $Y_i$. We wish to show this descends to map $Y_1 \sqcup_G Y_2$. If $G_i \subset Y_i$ is the preimage of the double locus $G$ inside $Y_i$, then it is naturally endowed with an isomorphism $\tau' : G_1 \to G_2$ such that $\tau'(\mathrm{Diff}_{G_1}(F)) = \mathrm{Diff}_{G_2}(F)$ where $F$ is the divisor of marked fibers on $Y_1 \sqcup_G Y_2$ making it a stable pair. Now $\mathrm{Diff}_{G_i}(F)$ consists of the point $F|_{G_i}$, the basespoint of the elliptic pencil on $Y_i$, which lies over $p_i$, as well as a contribution from the singularities of $Y_i$ along $G_i$. We described these singularities above in terms of the ramification data $(C, l)$, and from this description it is clear that the contribution of these singularities to the different is exactly given by the preimage of $\mathrm{Diff}_{l_i}(C_i)$. We conclude that $\tau'$ identifies the preimages of $p_i$ as well as of $\mathrm{Diff}_{l_i}(C_i)$. Thus $\tau'$ is the lift to $G_i \cong \bP^1$ of the unique isomorphism $\tau$ used to glue the two components in $\bP^2 \cup_l \bP^2$. Therefore two points on $G_1$ and $G_2$ map to the same point in $\bP^2 \cup_l \bP^2$ if and only if they are identified by $\tau'$ and so $\varphi : Y_1 \sqcup Y_2 \to \bP^2 \cup_l \bP^2$ factors through the broken degree one del Pezzo surface $X = Y_1 \cup_G Y_2 \to \bP^2 \cup_l \bP^2$ as claimed (we again use Koll\'ar's gluing theory; see \cite[Chapter 5.5]{singmmp}). 

Finally, consider $Q := \bP^2 \cup_l \bP^2 \subset \bP^3$ embedded as a singular quadric surface. We have a hyperplane section $H$ of $Q$ consisting of two lines meeting at the cone point $p$ (the image of the $p_i$ above). The pullback $\varphi^*H$ to $X$ consists of two elliptic pseudofibers, one on each component $Y_i$, meeting at the point on $G$ lying above $p$. If we further pullback to a component $Y_i$, we obtain the class of a pseudofiber, i.e. an element of the elliptic pencil on $Y_i$. On the other hand, denoting by $\nu_i : Y_i \to X$ the natural map, we have \[\nu_i^*K_X = K_{Y_i} + G_i = K_{Y_i} + 1/2f = -f + 1/2f = -1/2f,\] where $f$ is a pseudofiber class on $Y_i$ (see e.g. \cite[Lemma 7.7]{rational}). Thus $$\nu_i^*\varphi^*H = f = \nu_i^*(-2K_X)$$ is an equality of Cartier divisor classes which implies that $$\nu^*\varphi^*H = \nu^*(-2K_X)$$ where $\nu : Y_1 \sqcup Y_2 \to Y_1 \cup_G Y_2 = X$ is the gluing map. Since $G$ is an integral projective curve, then the induced map on Picard groups is injective (see e.g. \cite{MO}) so we conclude that $\varphi^*H = -2K_X$. \end{proof}

\begin{table}[htp!]
    \centering
    \caption{In the notation of the theorem, the following table describes the dictionary between the singular fiber of a normal component $Y$ along the double locus and the singularities of the branch data along $D$ (see Remark \ref{rmk:table}).}\label{table:fibers}
    \begin{tabu}{|c|c|c|}
    \hline
 
    \text{Singularities of $Br$} & $Br \cap D$ & \text{Singular fiber $Y$}\\
    \hline
       smooth & transverse at 3 points & $\mathrm{I}_0^*$  \\
       smooth & transverse at 1 point; tangent at 1  point   &   $\mathrm{I}_1^*$  \\
    nodal & transverse at 1 point; multiplicity 2 at node &  $\mathrm{I}_2^*$   \\
       cuspidal  & transverse at 1 point; multiplicity 2 at the cusp &   $\mathrm{I}_3^*$  \\
       $\mathrm{A}_2$ & transverse at 1 point; multiplicity 2 at this singular point & $\mathrm{I}_4^*$  \\
    %   \hline
    %   Isotrivial $j$-invariant $\infty$ components & & \\
    %   \hline
   %  Union of a double line and a line & $D$ cannot be contained %in a component of $Br$ & $\mathrm{N}_1$ \\
    %   \hline

    \hline
    \end{tabu}

\end{table}

\begin{remark}[see Table \ref{table:fibers}]\label{rmk:table} Note that $\mathrm{I}_n^*$ for $n > 4$ do not appear. If the surface component is not normal, then it is a $j$-invariant infinity component. In this case, $Br$ is the union of a non-reduced double line and a line. The double locus $D$ cannot be contained in a component of $Br$, and there has to be a marked fiber (with some multiplicity) passing through the point where the two components of $Br$ intersect. This is precisely where the singular $\mathrm{N}_1$ fiber lies. The other marked fibers correspond to other lines, except these cannot be contained in the double locus. Finally, we note that there is always at least one point where $D \cap Br$ is transverse -- these points must match up when gluing the two components of the singular quadric surface. If $D \cap Br$ is transverse at three points, it must be on both sides, giving two normal components of the del Pezzo surface glued along $\mathrm{I}_0^*$ fibers. Otherwise, there is a unique point where $D \cap Br$ is transverse. In this case, both sides have a unique point that matches up, and the remainder of the intersection points of $D \cap Br$ have multiplicity two -- these points are also identified.
\end{remark}

\subsection{An example}
 We give an example of the gluing construction for the base of the double cover structure on a surface of type $B_{\mathrm{II}}$ (see Figure \ref{fig:example}).

 \begin{figure}[htp!]
  \centering
 \includegraphics[scale=.75]{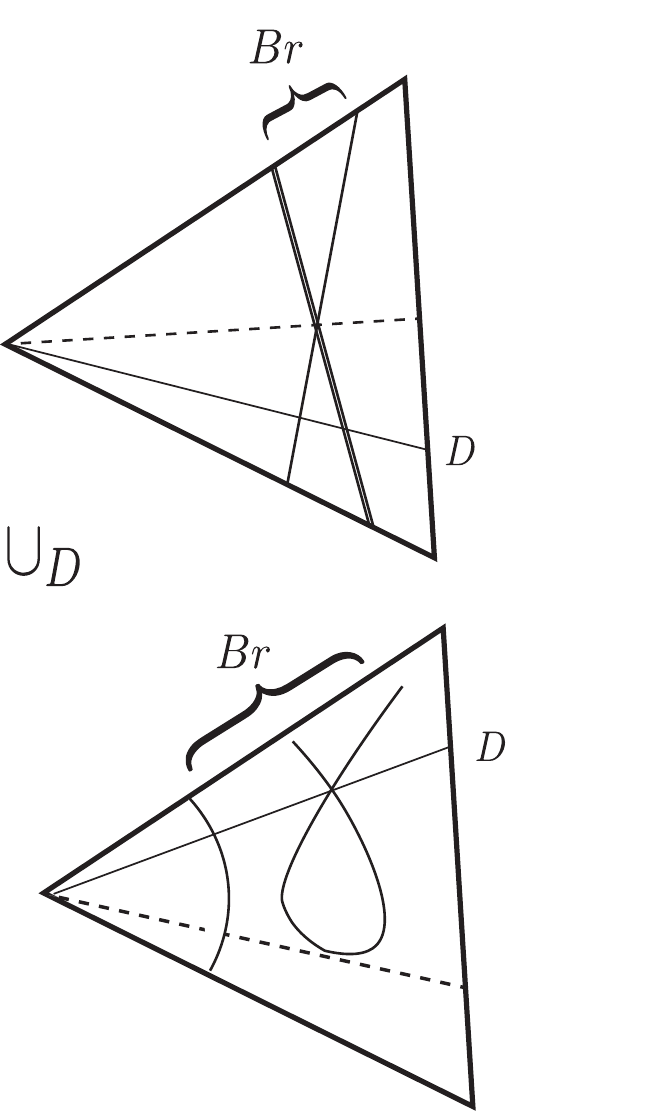}
 \caption{This figure shows the two $\bP^2$ components which glue together to give the degenerate quadric cone. The divisor $D$ denotes the double locus, and $Br$ denotes the (reducible) branch locus.}\label{fig:example} \label{fig:example}
 \end{figure}

 The two components of the base (both $\bP^2$) are glued along the line $D$, and the branch loci are denoted by $Br$. Note that the components of the branch loci must match up with the components of the same multiplicity along D on each surface. The dotted lines correspond to marked fibers: on the top component there is an $\mathrm{N}_1$ cusp and on the bottom component an $\mathrm{I}_1$. The non-reduced branch locus on the top surface corresponds to the double cover being isotrivial $j$-invariant $\infty$ with self intersection above the non-reduced component, and the fact that $D$ goes through a node on $Br$ on the bottom surface means that the gluing fiber above $D$ is a type $\mathrm{I}_2^*$ fiber. In particular, this depicts the branch data for a surface of type $B_{\mathrm{II}}$ corresponding to an isotrivial $j$-invariant $\infty$ component glued by twisted $\mathrm{N}_1/\mathrm{I}_2^*$ fibers to a normal $\mathrm{I}_2^*/4\mathrm{I}_1$ surface component.

\section{Proof of the main theorem} 

We are now ready to prove the following. 

\begin{theorem}\label{thm:maintheorem} The isomorphism $\calD^{1, s} \to \calQ_3^{s}$ extends to a separated morphism $\calR \to \overline{\calQ}_3$ such that the relative coarse map $\calR^c \to \overline{\calQ}$ is a monomorphism. \end{theorem} 

\begin{proof} For any surface pair $(X,F)$ parametrized by $\calR$, we have $H^1(X, \calO_X(-2K_X)) = 0$. Indeed for degree one del Pezzo surfaces, we saw this in Section \ref{sec:background}, and for broken del Pezzos this follows from \cite[Theorem 5.1]{master} (a vanishing theorem) and also \cite[Theorem 1.7]{fujino}. Thus given any family $\pi : (X, F) \to T$ of broken degree one del Pezzo surfaces, the vector bundle $\pi_*\calO_X(-2K_{X/T})$ is locally free and its formation is compatible with basechange. In particular, taking $T$ to be the spectrum of a DVR with generic fiber $X_\eta$ a smooth degree one del Pezzo and central fiber $X_0$ a broken surface of type $B_{\mathrm{I}}$, $B_{\mathrm{II}}$, or $B_{\mathrm{III}}$, we see that 
$$
\dim H^0(X_0, \calO_{X_0}(-2K_{X_0})) = 4.
$$

In Theorem \ref{thm:doublecover}, we constructed an explicit rank four sublinear series $V \subset H^0(X_0, \calO_{X_0}(-2K_{X_0}))$ such that $\varphi_V : X_0 \to \mathbb{P}^3$ is the double cover of the singular quadric cone $Q = \bP^2 \cup_l \bP^2$ as described above. By a dimension count this $V$ must be all of $H^0(X_0, \calO_{X_0}(-2K_{X_0}))$ so we conclude that the complete linear series $|-2K_{X_0}|$ induces the double cover map described in Theorem \ref{thm:doublecover} for any broken del Pezzo surface of type $B_{\mathrm{I}}$, $B_{\mathrm{II}}$, or $B_{\mathrm{III}}$. In particular, $\calO_{X_0}(-2K_{X_0})$ is globally generated so in families $\pi : (X, F) \to T$, the line bundle $\calO_X(-2K_{X/T})$ is $\pi$-generated and the surjection
$$
\pi^*\pi_*\calO_X(-2K_{X/T}) \to \calO_X(-2K_{X/T})
$$
induces a map $\varphi : X \to \bP(\pi_*\calO_X(-2K_{X/T})) =: \bP$ to a $\bP^3$-bundle over $T$ which is a double cover over a family of quadric cones $\calQ \subset \bP$ branched over the complete intersection of $\calQ$ with a cubic. This gives a morphism of algebraic stacks $\calR \to \calQ_3$ and $\calR$ is constructed as a substack of $\calM_v$ thus we get a map
$
\calR \to \calQ_3 \times \calM_v
$
extending the graph $\calD^{1,s} \cong \calQ_3^s \to \calQ_3 \times \calM_v$. Since $\calD^{1,s}$ is dense and open in $\calR$, this morphism must factor through the closure $\overline{\calQ} \subset \calQ_3 \times \calM_v$, giving the claimed morphism $\calR \to \overline{\calQ}$. Equivalently, $\overline{\calQ}$ is the scheme theoretic image of the morphism $\calR \to \calQ_3 \times \calM_v$.

Since $\calR$ is separated, so is the morphism $\calR \to \overline{\calQ}$. Moreover, $\calR$ is Deligne-Mumford so the inertia stack is quasifinite. In particular, it follows that the relative inertia stack is in fact \emph{finite}. Let $\calR^c \to \overline{\calQ}$ be the relative coarse moduli space which exists for a morphism of algebraic stacks with finite relative inertia (see \cite[Theorem 3.1]{aov}). Now, by the construction in Theorem \ref{thm:doublecover}, the surface $X$ is determined by branch locus in $Q$ and the marked anti-canonical curves on $X$ are determined by the marked lines $l$ on $Q$. Therefore the representable map $\calR^c \to \overline{\calQ}$ induces an injection on points and thus is a monomorphism. \end{proof}

\bibliographystyle{alpha}
\bibliography{main}

\end{document}